\documentclass[10pt]{amsart}
\usepackage{graphicx}
\usepackage{amsmath}
\usepackage{wasysym}
\usepackage{amscd}
\usepackage{endnotes}
\usepackage[active]{srcltx}
\usepackage[stable]{footmisc}
\usepackage{float}
\usepackage{caption}
\usepackage{wrapfig}
\usepackage{pdfpages}

\restylefloat{figure}

\usepackage[matrix, arrow,all,cmtip,color]{xy}

\usepackage{hyperref}

\restylefloat{figure}

\usepackage{headers-note}
\usepackage{epigraph}
\usepackage[margin=0.85in]{geometry}

\makeatletter
\usepackage{amssymb, stmaryrd}
\makeatother

\CompileMatrices

\newtheorem{theorem}[thm]{Theorem}

\def\rightc{\xrightarrow{(c)}}

\def\rightf{\xrightarrow{(f)}}

\def\rightw{\xrightarrow{(w)}}

\def\rightwc{\xrightarrow{(wc)}}

\def\rightwf{\xrightarrow{(wf)}}

\def\0{\varnothing}
\def\rtt{\rightthreetimes}
\def\lra{\longrightarrow}

\def\ilim{{\raisebox{0pt}{$\bigcirc$}} \kern -0.31cm \hbox{$\Yleft$}}

\DeclareMathAlphabet{\mathpzc}{OT1}{pzc}{m}{it}

\def\Qt{{\mathcal Qt}}
\def\QtN{\textrm{QtNaamen}}
\def\Qtc{\QtN_c}

\def\FC{\mathfrak C}
\def\Bx{B\times B}

\title[The UA in posetal MC]{The Univalence Axiom in posetal model categories}
\author[M. Gavrilovich]{Misha Gavrilovich} 
\author[A. HASSON]{Assaf Hasson$^*$}
\thanks{$^{*}$ Partially supported by Israel Science Foundation grant number 
XYZ}
\address{Department of mathematics\\
Ben Gurion University of the Negev\\
Be'er Sheva,\\
Israel} \email{hassonas@math.bgu.ac.il}
\author[I. Kaplan]{Itay Kaplan}
\address{
WWU M\:unster\\
Mathematisches Institut\\
Einsteinstraße 62\\
48149 Münster\\
Germany}
\email{itay.kaplan@uni-muenster.de}

\begin{document}\large
\maketitle

\begin{abstract}
In this note we interpret Voevodsky's Univalence Axiom in the language of 
(abstract) model categories. We then show that any posetal locally Cartesian 
closed model category $\Qt$ in which the mapping $Hom^{(w)}(Z\times B,C):\Qt\lra 
Sets$ is functorial in $Z$ and represented in $\Qt$ satisfies our homotopy 
version of the Univalence Axiom, albeit in a rather trivial way. This work was 
motivated by a question reported in \cite{Ob}, asking for a model of the 
Univalence Axiom not equivalent to the standard one. 
\end{abstract}

\section{introduction}

Though the notion of a categorical model of dependent type theory was known for 
quite some time now, it is only in recent years that it was realized that the 
extra categorical structure required to model the structure  of equality in 
dependent type theory corresponds to the structure of weak factorization 
equivalence, occurring in Quillen's model categories (\cite[p.2]{Ob}). This 
connection is the basis for V. Voevodsky project known as \emph{univalent 
foundations} whose main objective is to give a foundation of mathematics based 
on dependent type theory, which is intrinsically homotopical, in which types 
are interpreted not as sets, but rather as homotopy types (cf.). A central 
ideal in Voevodsky's univalent foundations is the extension of Martin-L\"of's 
dependent type theory by a ``homotopy theory reflection principle'', known as 
the \emph{Univalence Axiom}. Roughly speaking, the Univalence Axiom is the 
condition that the identity type between two types is naturally weakly 
equivalent to the type of weak equivalences between these types (V. Voevodsky, 
talk at UPENN, May 2011). 

Within the category (or, rather, the model category) of simplicial sets 
$sSets$, Voevodsky constructs a model of Martin-L\"of dependent type theory, 
satisfying also the Univalence Axiom. The models constructed in this way are 
called the standard univalent models (cf. Definition 3.2). During a 
mini-workshop around these developments held in Oberwolfach in 2010 the 
following question was raised: ``Does UA have models in other categories (e.g., 
1-topoi) not equivalent to the standard one?'', \cite[p.27]{Ob}. Though this 
question is probably referring to a univalent universe (for type theory), it 
seems to be meaningful also if taken literally.  It turns out that the 
Univalence Axiom can be given a precise meaning in the framework of Quillen's 
model categories (provided they are locally Cartesian closed). It is then 
meaningful to ask whether such a model category satisfies the Univalence Axiom. 

There are two main parts to this note. In the first of these parts (Section 
\ref{univ}) we give an interpretation of the notion of a univalent fibration in 
a purely category theoretic language. To formulate this notion we introduce, 
for a model category $\FC$, a correspondence $Hom^{(w)}(Z\times B,C): \FC\lra 
Sets$, intended to capture the class of weak equivalences between given fibrant 
objects $B,C\in \Ob\FC$. We then show that, given a fibration $p:C\rightf B$, 
if $Hom^{(w)}_{\Bx}(-\times B\times C,C\times B)$ is a representable functor 
(in the slice category $\FC/\Bx$), the ``obvious'' morphism (in $\FC/\Bx$) from 
the diagonal $B_\delta$ to $Hom_{\Bx}(B\times C,C\times B)$ factors 
``naturally'' (and uniquely in that sense) through the object representing this 
functor, $((C\times B)^{B\times C})_w$. We can then define the fibration $p$ to 
be \emph{univalent} if the morphism $m:B_{\delta}\lra ((C\times B)^{B\times 
C})_w$ is a weak equivalence. This construction (or a closely related one) is 
probably known to experts in the field, but since we could not find any 
reference  suitable for our purposes we give it in textbook detail. 

We then introduce the notion of a locally (w/f)-Caretsian closed model 
category, which is a locally Cartesian closed model category with the 
additional property that $Hom^{(w)}_{\Bx}(-\times B\times C,C\times B)$ is a 
representable functor for any fibrant objects $B,C$ and fibration $p:C\lra B$. 
We observe that in a posetal (w/f)-Cartesian closed model category all 
fibrations are univalent in the above sense. This is, of course, to be expected 
in view of Voevodsky's informal description of a univalent fibration as 
``...one of which every other fibration is a pullback in at most one way (up to 
homotopy)''. Apparently, this should suffice to assure that any posetal 
(w/f)-Cartesian closed model category satisfies the Univalence Axiom. But this 
does not provide, to our taste, a satisfying analogy with Voevodsky's 
construction. Such an analogy should have a natural interpretation of all the 
key features in Voevodsky's construction. To our understanding one such feature 
of univalent models is that they come equipped with a universal (univalent) 
fibration, of which all ``small'' fibrations are a pullback (in a unique way), 
\cite[Theorem 3.5]{Voev}. So our aim is to show, in addition, that such a 
universal fibration exists in our model category (with respect to an 
appropriate notion of smallness). 

The second of the main parts of the paper (Section \ref{qtc}) is dedicated to a 
self-contained construction of a posetal locally (w/f)-Cartesian closed model 
category, $\Qtc$. This construction is a special case of a more general 
construction introduced in \cite{GaHa}. In \cite[p.8]{Ob} Voevodsky writes: 
``Now for any $A, B : U$ , it is possible to construct a term $\theta : pathsU 
(A, B) \lra weq(A, B)$... The Univalence Axiom states that the map $\theta$ 
should itself be a weak equivalence for every $A, B : U$''. The map $\theta$ in 
Voevodsky's quote corresponds (to the best of our understanding) to the 
morphism $m$ appearing in the factorisation of the ``obvious morphism'' 
mentioned above. It is now obvious that this formulation of the Univalence 
Axiom is satisfied (in a rather trivial sense) in $\Qtc$. To fulfill our goals, 
it remains to construct an analogue in $\Qtc$ of Voevodsky's universe of 
``small'' fibration (those fibrations all of whose fibers are of cardinality 
smaller than $\alpha$ for some cardinal $\alpha$). To that end we suggest a 
(possibly over-simplified) notion of smallness for fibrations in a posetal 
model category, and show that with this definition $\Qtc$ admits a universe of 
small fibration (which is automatically univalent). 

Admittedly, the model category $\Qtc$ may be too simple an object to be of real 
interest. In \cite{GaHa} we suggest a construction of a (c)-(f)-(w)-labelled 
category analogous to that of $\Qtc$ resulting in a non-posetal category whose 
slices are equivalent to those of $\Qtc$. This category satisfies axioms 
(M1)-(M5) of Quillen's model categories, but does not have products (and 
co-products). We ask whether this richer category can be embedded in a model 
category, and whether such a model category would satisfy the Univalence Axiom, 
as formulated in this note. 

It should be made clear that none of the authors of this note is familiar with 
type theory and its categorical models. When we realized, moreover, that 
formally accurate literature on Voevodsky's univalent foundations exists only 
in the form of Coq code, we decided to base our homotopy theoretic 
interpretation of the Univalence Axiom on the somewhat less formal presentation 
appearing, e.g., in \cite{Voev}, \cite{Ob} and similar sources whose language 
is closer to the categorical language for which we were aiming. To compensate 
for the lack of precise references, we have taken some pains to give a detailed 
formal account of our interpretation of those sources. M. Warren's comments and 
clarifications, \cite{Warren},  were of great help to us, but all mistakes, are 
- of course - ours. 

A couple of words concerning terminology and notation are in place. In this 
text we refer to Quillen's axiomatization of model categories, as it appears in 
\cite{Qui}. Our usage of ``Axiom (M0)$\dots$(M5)'' refers to Quillen's 
enumeration of his axioms  in that book. Our commutative diagram notation is 
pretty standard, and is explained in detail in \cite{GaHa}. The labeling of 
arrows, (c) for co-fibrations, (f) for fibrations and (w) for weak 
equivalences, is borrowed from N. Durov. 

\section{Cartesian closed posetal categories}\label{CCC}
Given a category $\FC$ and $B,C\in \Ob\FC$, it is often desirable to treat 
$Hom(B,C)$ as an object of the category: this is a natural requirement, as it 
is inconvenient, while working in $\FC$, to be constantly required to work with 
elements 
external to $\FC$, namely, working with  $Hom$-sets merely as sets. A category 
$\FC$ is Cartesian closed if it is closed under ``exponentiation'', namely, 
that given $B,C\in \Ob \FC$, an object $C^B$ (satisfying certain category 
theoretic properties to be explained shortly) exists. As notation suggests, the 
object $C^B$ is supposed to represent the set $Hom(B,C)$. We will now explain 
this in more detail: 

A category $\FC$ (with finite limits, or at least binary products)  is called 
{\em Cartesian closed} if
for every $B,C\in \Ob\FC$  there exist an object, denoted $C^B$, 
and an arrow $\epsilon:C^B \times B \lra C$ such that for every 
object $D$ and arrow $D\times B \xrightarrow{g} C$, there is a unique arrow 
$f:D\lra C^B$
such that $D\times B\xrightarrow {f \times \id_B} C^B \times B \xrightarrow 
{\,\epsilon\,} C$
and $D\times B \xrightarrow {\,g\,} C$  coincide \cite[6.2,p.108]{Awodey}. 

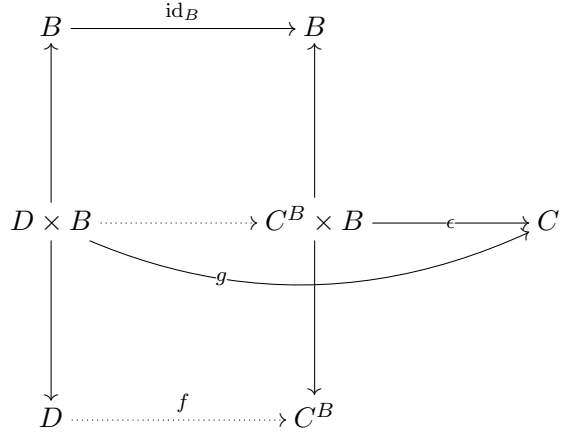
\begin{figure}[H]
\centerline{
\xymatrix @R=5pc @C=5pc{
B  \ar[r]^{\id_B} & B  \\
D\times B \ar@/_2pc/|-(0.3)g[rr] \ar@{.>}[r] \ar[u]\ar[d] & C^B\times B 
\ar[r]|-{\epsilon} \ar[u]\ar[d]& C  \\
D \ar@{.>}[r]^f & C^B 
}}
\caption{The existence of the arrow $D\xrightarrow{f} C^B$ assures the 
existence of the arrow $D\times B\xrightarrow{f\times \id_B} C^B\times B$ by 
the universal property of $C^B\times B$.}\label{CC}
\end{figure}

Given $B,C\in \FC$,  Figure \ref{CC} implies the existence of a bijective 
correspondence, functorial in $D$, between $Hom(D\times B,C)$ and $Hom(D,C^B)$, 
which we can write as: 
\begin{equation*}\tag{$*$}
 Hom(D\times B,C)\equiv Hom(D,C^B)
\end{equation*}

In the above equation we say that $C^B$ is an \emph{object representing} the 
functor $Hom(-\times B,C):\FC\lra Sets$, i.e. that this functor is naturally 
equivalent to the functor $Hom(-,C^B)$. This equation is of importance not only 
in understanding the ideology behind the definition of a Cartesian closed 
category, but will also play an important role in our interpretation of the 
Univalence Axiom in a model category. 

To see why $C^B$ can be, in many cases, identified with the set $Hom(B,C)$ 
consider, in the equation $(*)$ the terminal object, $\top$ (for the variable 
$D$). We get an equivalence of categories: 
\[
Hom(B,C)\equiv  Hom(\top\times B, C) \equiv Hom(\top, C^B)
\]
which, in many cases (e.g., in the category $Sets$, or in the category $Top$ 
where $\top$ is a point) gives: 
\[
 Hom(B,C)=Hom(\top,C^B)=C^B
\]
This explains the general category theoretic convention of identifying 
exponents with $Hom$-sets.

In this note we will be interested, mainly, in posetal model categories. We 
conclude this section with a discussion of a posetal category being Cartesian 
closed. Recall that a category $\FC$ is \emph{posetal} if arrows are unique 
whenever they exist. Namely, given $B,C\in \Ob\FC$ there exists at most one 
$f\in \Mor\FC$ such that $B\xrightarrow{f} C$. Thus, in a posetal category all 
diagrams are commutative, and therefore,  as can be seen in Figure \ref{CC}, if 
$\FC$ is posetal, to verify that $\FC$ is Cartesian closed it is enough to 
verify that for any $B,C\in \Ob \FC$  
there exists an object $C^B$ such that $C^B \times B \lra C$ and such that for 
every
object $D$,  $D\times B \lra C$ implies  $D\lra C^B$.

Given a category $\FC$ and $A\in \Ob\FC$, the \emph{slice of $\FC$ over $A$}, 
denoted $\FC/A$ is the category of arrows $B\lra A$: its objects are arrows 
$B\lra A$ in $\FC$ and an arrow from $B\lra A$ to $C\lra A$ is an arrow in 
$\FC$ making the triangular diagram commute. For a posetal category the slice 
$\FC/A$ can be identified with the full sub-category whose objects are all 
$B\in \Ob\FC$ such that $B\lra A$. 

A category is {\em locally Cartesian closed} if $\FC/A$ is 
Cartesian closed for all $A\in \Ob\FC$, \cite[Prop.9.20,p.206]{Awodey}. Observe 
that a posetal category with a terminal object is Cartesian closed if and only 
if it is locally Cartesian closed. Indeed, $\FC$ has a terminal object $\top$ 
and $\FC/\top$ --- which is, by assumption, Cartesian closed - is merely $\FC$, 
so locally Cartesian closed implies Cartesian closed. In the other direction, 
if $\FC$ is Cartesian closed and $A\in \Ob\FC$ is any object, $B,C\in \Ob\FC/A$ 
then $C^B\times A\lra A$. Thus, $C^B\times A\in \Ob\FC/A$. So it remains to 
verify that for any $D\in \Ob\FC/A$, if there exist an arrow $D\times B\lra C$ 
then there exist an arrow $D\lra C^B\times A$. By definition, there is an arrow 
$D\lra C^B$, and since $D\in \Ob\FC/A$ there is an arrow $D\lra A$. By the 
universal property of $C^B\times A$ this means that there is an arrow $D\lra 
C^B\times A$, as required. 

\section {The Univalence Axiom}\label{univ}
As explained in the introduction, the original formulation of the Univalence 
Axiom is given in the language of type theory (and, apparently, its precise 
formulation exists only in Coq code). The axiom asserts that, given a universe 
of type theory, the homotopy theory of the types in this universe should be 
fully and faithfully reflected by the equality on the universe. To prove that 
the universes of type theory he constructs in the category $sSets$ are 
univalent, Voevodsky proves, \cite[Theorem 3.5]{Voev}, that there is a 
fibration universal for the class of \emph{small} fibrations, and that this 
fibration is \emph{univalent}. Apparently, this statement is the right 
reformulation of the Univalence Axiom in the context of the model category of 
simplicial sets. 

In order to generalize the Univalence Axiom to arbitrary (locally Cartesian 
closed) model categories, one has to explain what it means for a fibration to 
be univalent, and to define a suitable notion of smallness.  Our first step is 
to define (and explain) what is a univalent fibration in an arbitrary locally 
Cartesian closed model category. We then show using this definition, that if 
our locally Cartesian closed model category, $\FC$, is posetal, then every 
fibration is univalent. Thus, to show that such a model category $\FC$ meets 
the Univalence Axiom (for a suitable notion of smallness) it remains to show 
that a universal fibration for all small fibrations exists. This section is 
concluded with the observation that this is indeed the case for a natural 
(though somewhat trivial) notion of smallness, provided $\FC$ is posetal. 

\subsection{A model category object for weak equivalences} 
Recall that Voevodsky's formulation of the Univalence Axiom takes place in the 
category of simplicial sets. In order to reformulate this axiom in the more 
general setting of model categories we have to set up a dictionary between 
Voevodsky's terminology and the common terminology of model categories. 
Apparently, such a translation is folklore to the experts, but since we were 
unable to find a precise formulation meeting the level of generality need for 
this note, we give the details. The main difficulty in this translation is the 
definition of a univalent fibration. Since there is no literature on the 
subject, our translation of this notion relies almost entirely on Voevodsky's 
notes, \cite{Voev}, and some clarifications corresponded to us by Warren, 
\cite{Warren}. 
 
Let us recall Voevodsky's definition of a univalent fibration in the category 
$sSets$ of simplicial sets, \cite[p.7]{Voev}: 

\begin{quote}
For any morphism $q : E \rightarrow B$ consider the simplicial set 
$\underline{Hom}_{B\times B} (E \times B, B \times E)$. If $q$ is a
fibration then it contains, as a union of connected components, a simplicial 
subset 
$weq(E\times B, B\times E)$ which corresponds to morphisms which are weak 
equivalences. 
The obvious morphism from the diagonal $ \delta δ : B \rightarrow B \times B$ 
to 
$\underline{Hom}_{B\times B} (E \times B, B \times E)$ over $B \times B$ 
factors uniquely through a morphism 
$m_q : B \rightarrow weq(E \times B, B \times E)$. 
\end{quote}

In this terminology the fibration $q:E\lra B$ is \emph{univalent} if the 
morphism $m_q:B\lra weq(E\times B,B\times E)$ is a weak equivalence (cf. 
Definition 3.4 [ibid.])

Voevodsky's text translates readily into the language of Cartesian closed model 
categories, with the possible exception of the definition of the object 
$weq(E\times B,B\times E)$. In this subsection we perform this translation, 
focusing on the model categorical definition of $weq(E\times B,B\times E)$. As 
we will see, the object $weq(C,B)$ has much in common with the exponential 
$C^B$, it is therefore convenient to introduce: 
\begin{notation}
 Given a model category $\FC$ and $B,C\in \Ob\FC$, the object $weq(C,B)$ will 
be denoted $C^B_w$. 
\end{notation}

For the sake of clarity, we explain the above text word for word. So let $\FC$ 
be a locally Cartesian closed model category, $E,B\in \Ob\FC$ and $q:E\lra B$ a 
fibration. Let $E\times B$ be the product of $E$ and $B$ in $\FC$. This objects 
comes with two morphisms: $E\times B\xrightarrow{\pr_E^{E\times B}} E$ and 
$E\times B\xrightarrow{\pr_{B}^{E\times B}} B$. 
Since a morphism into a product is uniquely determined by a pair of morphisms 
into its components the following defines 
a unique morphism: 
$(q,id):E\times B \xrightarrow{q\circ \pr_E^{E\times B}\,\times\, 
{\pr_{B}^{E\times B}}} B\times B$. 
In set-theoretic notation, the morphism $q\circ \pr_E^{E\times B}\,\times\, 
{\pr_{B}^{E\times B}}$ defined above is given by the mapping $(e,b)\mapsto 
(q(e),b)$.

As an object of $\FC/B\times B$ this morphism is denoted by Voevodsky  $E\times 
B$. 
In order to define the object $B\times E\in \Ob\FC/B\times B$ observe that (in 
$\FC$) the object $B\times B$ comes equipped with two morphisms $\pr_1,\pr_2$ 
into each of its components. Thus, there is a morphism (in $\FC$) 
$\tau_{B\times B}:B\times B \xrightarrow {\pr_2\times
\pr_1} B\times B$ (which can be thought of us the morphism permuting the 
factors of the product).  In set-theoretic notation $\tau_{B\times 
B}(b_1,b_2)=(b_2,b_1)$. 
The morphism $\tau\circ (q,id)$ as an object of $\FC/B\times B$ is denoted by 
Voevodsky $B\times E$. 

Thus, we have interpreted $B\times E$ and $E\times B$ as objects in the slice 
category $\FC/B\times B$. In view of our discussion of exponentials in Section 
\ref{CCC}, this allows us to identify $\underline{Hom}_{B\times B}(E\times 
B,B\times E)$ with the object $((B\times E)^{E\times B})_{B\times B}$ (recall 
that our assumption that $\FC$ is locally Cartesian closed assures that such an 
object exists). 

Recall that, in Voevodsky words, ``[the object $weq(E\times B, B\times E)$]
corresponds to morphisms [i.e., elements of $Hom_{B\times B}(E\times B, B\times 
E)$] 
which are weak equivalences [in the slice category $\Qt/B\times B$]''. Let us 
now try to understand, in more generality, given a Cartesian closed model 
category $\FC$ and objects $B,C\in \Ob\FC$ what should be the object $C^B_w$. 
In the terminology used in Section \ref{CCC} Voevodsky's text  should mean that 
the object $C^B_w$ \emph{represents} the set of morphisms from $B$ to $C$ 
satisfying the additional requirement that these morphisms are weak 
equivalences. In a Cartesian closed (model) category we identified $Hom(B,C)$ 
with the functor $Hom(-\times B,C)$. Since, in Voevodsky's text $C^B_w$ is a 
sub-object of the exponential $C^B$, it is natural to try and identify $C^B_w$ 
with a sub-functor, let us denote it $Hom^{(w)}(-\times B,C)$, of $Hom(-\times 
B,C)$. Moreover, any $Z\in \Ob\FC$ and morphism $f:B\lra C$ induces a morphism 
$Z\times B\lra C$ given by $f\circ \pr_B^{Z\times B}$. It is, therefore, 
reasonable to require that the same be true for the sub-functor 
$Hom^{(w)}(-\times B,C)$. The ``obvious'' choice of letting $Hom^{(w)}(Z\times 
B,C)$ be the set of all morphisms $h:Z\times B\rightw C$ does not have this 
property. So the next best choice seems to be: 

\begin{notation}
 Given $Z,B,C\in \Ob\FC$, let 
\[
 Hom^{(w)}(Z\times B,C):=\{h:Z\times B\lra C\:|\: (\pr_Z^{Z\times B}\times 
h):Z\times B\rightw Z\times C\}.
\]
\end{notation}
Observe that $Hom^{(w)}(Z\times B,C)\subseteq Hom(Z\times B,C)$ for all $Z$. We 
do not know, however, whether --- in general --- it is functorial in $Z$. We 
leave it as an exercise to the reader to show that if $\FC$ is right proper 
(i.e., if the base change of a weak equivalence along a fibration is again a 
weak equivalence), then $Hom(-\times B,C)$ is indeed functorial provided that 
$B\rightf \top$ and $C\rightf \top$. We remind that $sSets$ is right proper, 
and so is $Top$ --- and, more generally, any model category all of whose 
objects are fibrant (see below) is right proper. 

At all events, if $Hom^{(w)}(-\times B,C)$ is a functor, and as such it is 
represented in $\FC$ we let $C^B_w$ denote the representing object. In 
particular we obtain: 
\begin{equation*}\tag{$**$}
 Hom^{(w)}(Z\times B,C)\equiv Hom(Z,C^B_w)
\end{equation*} 

\begin{definition}
 Let $\FC$ a model category. Say that $\FC$ is (w/f)-Cartesian closed, if it is 
Cartesian closed and, in addition,  $Hom^{(w)}(-\times B,C):\FC\lra Sets$ is 
represented (in the sense of $(**)$ above) for all fibrant $B,C\in \Ob \FC$ 
(i.e., the morphisms $B\rightf \top$ and  $C\rightf \top$ into the terminal 
object, $\top$ are fibrations). Say that $\FC$ is locally (w/f)-Cartesian 
closed, if for any $X\in \Ob\FC$ the slice category $\FC/X$ is (w/f)-Cartesian 
closed. 
\end{definition}

\begin{rem}
 The above definition is our straightforward interpretation of Voevosky's words 
in the langauges of Cartesian closed model categories. This definition is 
sufficeint for our purposes, as it is met by the model category $\Qtc$, 
constructed in the last section of this note. It is conceivable that, in the 
general setting, a more accurate definition will be required. 
\end{rem}

We will now show that if $\FC$ is a (locally) Cartesian closed model category 
such that  $Hom^{(w)}_{B\times B}(-\times (E\times B), B\times E)$ is 
represented in $\FC/B\times B$ then the object representing this functor 
satisfies the requirement in Voevodsky's text, namely, the ``obvious'' morphism 
from the diagonal to $\underline{Hom}_{B\times B}(E\times B,B\times E)$ factors 
uniquely through this representing object. 

Before we proceed, some explanations are needed. Recall that we are working in 
the slice category $\FC/\Bx$. Thus, the diagonal $\delta: B\to \Bx$ is an 
object of $\FC/\Bx$, which we denote $B_{\delta}$. To avoid confusion, we 
denote $E\underline\times B:=E\times B$ and $B\underline\times E:=B\times E$ 
(viewed as objects of $\FC/\Bx$, as explained above). We shall also let, given 
an object $X\in \Ob\FC/\Bx$, $X_s\in \Ob\FC$ denote the source object of the 
morphism (in $\FC$) corresponding to $X$. We have already explained in Section 
\ref{CCC} in what sense  $\underline{Hom}_{\Bx}(E\underline\times 
B,B\underline\times E)$ can be viewed as an object of $\FC/\Bx$. So in order to 
make Voevodsky's statement clear we only have to explain what is the ``obvious 
morphism'' from the diagonal to  $\underline{Hom}_{\Bx}(E\underline\times 
B,B\underline\times E)$.

Consider the product $B_\delta  \times_{B\times B} E\underline\times B$, 
by definition it is the pullback of the morphisms 
$B \xrightarrow \delta B\times B$ and $E\times B \xrightarrow{(q,id)} B\times 
B$. 
\begin{figure}[H]
\centerline{
\xymatrix @R=2pc @C=2pc{
(B_\delta \times_{\Bx} E\underline\times B)_s \ar[r]_(0.6){\pr_2} 
\ar[d]^{\pr_1} & E\times B \ar[d]^{(q,\id_B)} \\
B\ar[r]^{\delta} \ar@/_2pc/[rr]|-\delta & B\times B\ar[r]^{\tau} & \Bx}}
\caption{}\label{2mor}
\end{figure}

It follows immediately from the fact that $\delta$ is the diagonal morphism and 
from the definition of $\tau$ that $\tau\circ \delta\circ \pr_1=\delta\circ 
\pr_1$. The right hand side morphism in the above equality corresponds in 
$\FC/\Bx$ to the object $B_\delta\times (E\underline\times B)$, while the 
composition $\tau\circ (q,\id_B)$ corresponds, by definition, to the object 
$B\underline\times E$. The commutativity of the diagram of Figure \ref{2mor} 
implied by the above equality means, by definition of $\FC/\Bx$ that the 
morphism $\pr_2$ corresponds in $\FC/\Bx$ to a morphism  $h:B_\delta\times 
(E\underline\times B)\lra B\underline\times E$ in $\FC/\Bx$. By  $(*)$ the 
morphism $h$ corresponds to a morphism $\bar m_q:B_\delta \lra 
\underline{Hom}_{B\times B} (E\underline\times B, B\underline\times E)$.
The morphism $\bar m_q$ is {\em the obvious morphism from the diagonal $\delta 
δ : B \rightarrow B \times B$ to
$\underline{Hom}_{B\times B} (E\underline\times B, B\underline\times E)$ over 
$B \times B$}.

Let us denote $\pi_1$ and $\pi_2$ the morphisms in $\FC/\Bx$ corresponding to 
the morphisms $\pr_1$ and $\pr_2$ respectively (see Figure \ref{2mor}). Note 
that these morphisms can be identified with the two canonical morphisms 
associated to $B_{\delta}\times (E\underline\times B)$ as the pullback of 
$B_\delta$ and $E\underline\times B$. Thus, we have a morphism (in $\FC/\Bx$) 
$\pi_1\times(\tau\circ\pi_2):B_\delta\times (E\underline\times B) \lra 
B_\delta\times (B\underline\times E)$. This morphism, by definition of 
$\FC/\Bx$, arises from a morphism $\sigma: ((B_\delta\times (E\underline \times 
B))_s\lra ((B_\delta \times (B\underline \times E))_s$ in $\FC$. This morphism 
in $\FC$ is readily seen to be an isomorphism, and therefore a weak 
equivalence. By the definition of the model structure on $\FC/\Bx$, a morphism 
$X\lra Y$ in $\FC/\Bx$ is labelled (w), (f) or (c) if and only if the 
corresponding morphism $X_s\lra Y_s$ in $\FC$ is labelled (w), (f) or (c) 
respectively. Thus $\sigma$ (or, rather, $(\pi_1,\tau\circ \pi_2)$) is a weak 
equivalence also in $\FC/\Bx$. By definition this means precisely that $h\in 
Hom^{(w)}(B_\delta \times (E\underline\times B),B\underline\times E)$. 
Therefore, $(**)$ implies that $h$ corresponds to a unique morphism 
$m_q:B_\delta \lra ((E\underline\times B)^{(B\underline\times E)})_w$. 

Finally, observe that in any (w/f)-Cartesian closed model category, $\FC$, and 
for all fibrant $B,C\in \Ob\FC$ setting $D=C^B_w$ in $(**)$, we get that the 
identity $\id: C^B_w\lra C^B_w$  is an element of $Hom(C^B_w,C^B_w)$ and 
therefore also of $Hom^{(w)}(C^B_w\times B,C)$. Because $Hom^{(w)}(D\times 
B,C)\subseteq Hom(D\times B,C)$ for all $D$, we get that $\id: C^B_w\lra C^B_w$ 
is an element of $Hom(C^B_w\times B,C)\equiv Hom(C^B_w,C^B)$. Thus, $\id: 
C^B_w\lra C^B_w$ induces, through this last equivalence of functors, a morphism
$\id_*^{(B,C)} : C^B_w\lra C^B$
and a natural transformation $ \id_*^{(B,C)}:Hom^{(w)} (- \times B, C)\equiv
Hom(-,C^B_w)$ coinciding
with the natural transformation provided by the inclusion $Hom^{(w)}(- \times 
B, C)\subseteq Hom(- \times B, C)$.    
Combining this with the conclusion of the previous paragraph, we get that $\bar 
m_q=\id_*^{(B,C)}\circ m_q$. 
The identification of $\id: C^B_w\lra C^B_w$ as an element of $Hom(-\times 
B,C)$ is natural in that sense. Requiring that $\bar m_q$ factors 
\emph{naturally} through $C^B_w$ amounts, therefore, to the requirement that 
this factorization is obtained via $\id_*^{(B,C)}$. We observe that with this 
additional requirement this factorization is unique. 

\subsection{The Univalence Axiom in posetal model categories}

Having defined the object $C^B_w$ for a locally (w/f)-Cartesian closed model 
category $\FC$, we can define a fibration $p:E\lra B$ to be \emph{univalent} if 
the morphism $\bar m_q:B_\delta\lra ((E\underline\times B)^{(B\underline\times 
E)})_w$ is a weak equivalence. In this subsection we prove: 

\begin{lem}
 Let $\FC$ be a locally (w/f)-Cartesian closed posetal model category. Then 
every fibration is univalent. 
\end{lem}
\begin{proof}
 First, observe that since $\FC$ is posetal for any object $B\in \Ob\FC$ the 
product $B\times B$ is isomorphic to $B$. Indeed, by the universal property of 
$\Bx$ there is a morphism $B\lra B\times B$, and since $\Bx\lra B$ we get that 
$B\cong B\times B$ (because $\FC$ is posetal). 

Recall that since $\FC$ is posetal, for any object $B$ the slice category 
$\FC/B$ can be identified with the full subcategory whose objects are $\{A\in 
\Ob\FC: A\lra B\}$. Namely, the morphism $A\lra B$, as an object in $\FC/B$ can 
be identified with the object $A$ in $\FC$. In particular $B_\delta$ can be 
identified with the object $B$, and given a fibration $E\rightf B$ the objects 
$E\underline\times B$ and $B\underline\times E$ in $\FC/\Bx\equiv \FC/B$ are 
isomorphic and both can be identified with the object $E\times B$ of $\FC$ 
(indeed, in a posetal model category $E\underline\times B$ and 
$B\underline\times E$ are the same object since $(E\underline\times 
B)_s=(B\underline\times E)_s$). So $\underline{Hom}_{\Bx}(E\underline\times 
B,B\underline\times E)$ is isomorphic to the object $(E\times B)^{E\times B}$. 

It will suffice to show that for any object $C\in \Ob\FC$ the exponent $C^C$ is 
isomorphic to $\top$, the terminal object of $\FC$. Indeed, then 
$\underline{Hom}_{\Bx}(B\underline\times E, E\underline\times B)=((E\times 
B)^{(E\times B))}_{\Bx}=\top_{\Bx}$. But the terminal object of $\FC/\Bx$ is 
$\Bx=B$. We get that the ``obvious morphism'' $h:B_\delta \lra 
\underline{Hom}_{\Bx}(E\underline\times B,B\underline\times E)$ defined in the 
previous subsection corresponds to the arrow $B\lra B$, so it is an 
isomorphism, and therefore a weak equivalence. But in a posetal model category, 
if $X\lra Y$ is an isomorphism then for all $Z$, if $X\lra Z\lra Y$ then $X\lra 
Z$ and $Z\lra Y$ are both isomorphisms. In particular, the morphism $\bar 
m_q:B_\delta \lra ((E\underline\times B)^{(B\underline\times E)})_w$ is an 
isomorphism, and therefore a weak fibration. 

It remains, therefore, to show that in a posetal Cartesian closed model 
category $C^C\cong \top$ for all $C\in \FC$. Indeed, $\top\times C\cong C$, 
implying $\top\times C\lra C$. So by Figure \ref{CCC} (with $D=\top$ and $B=C$) 
we get an arrow $\top\lra C$, and $\FC$ being posetal we get $C^C\cong \top$. 
\end{proof}

Having seen that in posetal locally Cartesian closed model categories the 
notion of univalent fibrations degenerates, it remains to show that there 
exists a fibration $p$ universal for the class of \emph{small} fibrations. Of 
course, the notion of smallness in this context should be defined as well. 

\begin{definition}
Let $\FC$ be a model category, Fix a morphism $\tilde U \xrightarrow{p} U$. A 
morphism  $Y \xrightarrow{f} X$ is $p$-small if $Y\xrightarrow{f} X$ fits in a 
pull-back square: 

\begin{figure}[H]
\centerline{
\xymatrix @R=2pc @C=2pc{
Y \ar@{.>}[r] \ar[d]|-f & \tilde U \ar[d]|-p \\
X\ar[r]|-{f_p} & U
}}
\caption{This is a pullback square, if for any morphisms $Z\lra X$ and $Z\lra 
\tilde U$ making the diagram commute there is an arrow $Z\lra Y$ making the 
diagram commute. }
\end{figure}

 Say that $p$ is \emph{universal} (with respect to a pre-defined class of 
\emph{small} fibrations) if the class of $p$-small fibrations contains all 
small fibrations.  
\end{definition}

Observe that in a posetal category, given morphisms $p$ and $f$ as in the above 
definition, the morphism $X\xrightarrow{f_p} U$ is unique if it exists. 
Therefore, $Y\xrightarrow{f} X$ is $p$-small if and only if $X\lra U$ and 
$Y=\tilde U\times X$.

\begin{lem}\label{psmall}
Let $\Qt$ be a posetal model category. Consider the unique morphism $\0\lra 
\top$ and let $\tilde U$ be the unique object such that $\0\rightwc \tilde U 
\rightf \top$. Let $p$ denote the fibration $\tilde U \rightf \top$.  Assume, 
in addition, that all morphisms in $\Qt$ are co-fibrations. Then a fibration 
$f:Y\lra X$ is $p$-small iff $\emptyset \rightwc Y$.
\end{lem}
\begin{proof}
The key to the proof is the following observation: 

\noindent{\bf Claim} If $Z\lra \tilde U$ then $Z\rightwc \tilde U$.\\

\proof Let $Z\rightwc Z_{wc}\rightf \tilde U$  It will suffice to prove that 
$\tilde U\lra Z$, since then $Z_{(wc)}$ is isomorphic to $\tilde U$ (member 
that $\Qt$ is posetal).  Indeed, consider the following diagram: 
\begin{figure}[H]
\centerline{
\xymatrix @R=2pc @C=2pc{
\perp \ar[d]|-{(wc)} \ar[r] & Z_{wc} \ar[d]|-{(f)}\\
\tilde U \ar[r] \ar@{.>}[ur] & \tilde U
}}
\end{figure}
Finishing  the proof of the claim \qed$_{\text {Claim}}$

Now, if $\perp\rightwc Y\rightf X$ (where $\perp$ is the initial object), then 
$\perp\lra Y\rtt \tilde U \lra \top$, giving $Y\lra \tilde U$. So it suffices 
to show that $Y=\tilde U\times X$. Because $\perp \rightwc Y \rtt \tilde U 
\rightf \top$ we know that $Y\lra \tilde U$. So $Y\lra X\times \tilde U$. Let 
$Y\rightwc Y_{wc} \rightf X\times \tilde U$. By the above claim $X\times \tilde 
U \rightwc \tilde U$ and $Y_{wc}\rightwc \tilde U$. So by (M5): 
\begin{figure}[H]
\centerline{
\xymatrix @R=2pc @C=2pc{
& Y_{wc} \ar[dl]|-{\therefore (w)} \ar[dr]|-{(wc)}\\
X\times \tilde U \ar[rr]|-{(wc)} & & \tilde U 
}}
\caption{By (M5) the arrow $Y_{wc}\lra X\times \tilde U$ is a weak equivalence.}
\end{figure}
But, by assumption all arrows in $\Qt$ are co-fibrations, and we chose $Y_{wc}$ 
so that $Y_{wc}\rightf X\times \tilde U$. so $Y_{wc}\xrightarrow{(wcf)}X\times 
\tilde U$, and since $\Qt$ is posetal, this implies that $Y_{wc}$ is isomorphic 
to $X\times \tilde U$. We conclude that $Y\rightwc  X\times \tilde U$. 
Therefore $Y\lra X\times \tilde U \rtt Y\lra X$, giving an arrow $X\times 
\tilde U \lra Y$, with the conclusion that $Y$ is isomorphic to the product, as 
required. 

In the other direction. If $Y\rightf X$ is $p$-small then $Y\lra \tilde U$, and 
by the claim $Y\rightwc \tilde U$. Similarly, if $\perp \rightwc Y_{wc}\rightf 
Y$ then $Y_{wc}\rightwc \tilde U$. So (M5), applied to the triangle $\tilde 
U\longleftarrow Y_{wc}\lra Y\lra \tilde U$, assures that $Y_{wc}\rightw Y$. 
Since, by assumption, all arrows are co-fibrations, we get 
$Y_{wc}\xrightarrow{(wcf)} Y$, with the conclusion that $\perp\rightwc Y$, as 
required. 
\end{proof}

In order to conclude we have to give a reasonable notion of smallness --- 
namely, to define when is a fibration $f:X\to Y$ in an arbitrary model category 
\emph{small}. For reasons to be explained below we do not attempt to give a 
definition of a small fibration in that generality. Rather, our goal is find 
some (minimal) necessary conditions that such a class of fibrations should 
satisfy. Since we are trying to interpret the Univalence Axiom, as it is 
discussed in \cite{Voev}, it is natural that our analysis of the notion of 
smallness be based on the definition of a universal fibration introduced there 
(\cite{Voev}, p.6). In the category of simplicial sets a fibration $f:X\lra Y$ 
is \emph{small} (for some fixed cardinality $\alpha$) if all its fibers are of 
cardinality smaller than $\alpha$.

Observe that (for most cardinalities) Voevodsky's definition of small 
fibrations depends, e.g., on the choice of model of ZFC. This suggests that 
there is no natural category theoretic counterpart exactly capturing this 
definition. So, let us consider some obvious properties of Voevodsky's 
definition: 
\begin{enumerate}
 \item The class of small fibrations is closed under finite products and 
co-products.
\item Since co-fibrations are injective, 
if 
$f:X\lra Y$ is a small fibration, and $g:X'\lra X$ is a fibration and a 
co-fibration then also $f\circ g: X'\lra Y$ is small.
\end{enumerate}

 Thus, by the second point above, if $\Qt$ is a posetal model category all of whose 
morphisms are co-fibrations, then for any small fibration $X\rightf Y$ if 
$\perp\rightwc X_{wc} \rightf X$, then $X_{wc}\rightf Y$ should also be a small 
fibration. Therefore, in any such model category, under any non-trivial 
definition of small fibrations, some small fibrations will be of the form 
$X_{wc}\rightf Y$ where $\perp \rightwc X_{wc}$. Moreover, in order to satisfy 
the first of the above points we have to close the collection of trivial 
co-fibrant objects,  $X_{wc}$, such that there exists some small fibration $X_{wc}\rightf 
Y$ under finite limits and co-limits. The properties of $\Qt$ assure that this 
is still a collection of trivial co-fibrant objects (that the co-base change of 
a weak co-fibration is a weak equivalence - and therefore in $\Qt$ a weak 
co-fibration - follows from Axiom (M4) of model categories; that the product of 
two trivial co-fibrant objects is a trivial co-fibrant objects is proved 
precisely as in the claim of Lemma \ref{psmall}). Let $\mathcal S$ denote the 
collection of trivial co-fibrant objects thus obtained. Consider 
$\Qt^{\mathcal S}$, the ``co-slice category over $\mathcal S$'', i.e., 
$\Qt^{\mathcal S}$ is the full sub-category whose objects are all those $X\in 
\Ob\Qt$ such that $S\rightf X$ for some $S\in \mathcal S$. 
Then 
$\Qt^{\mathcal S}$ is still a model category (one only needs to check that 
$\Qt^{\mathcal S}$ is closed under finite limits and co-limits, which is 
obvious). Moreover, as can be readily checked in Figure 1, since $\Qt$ is 
Cartesian closed, so is $\Qt^{\mathcal S}$. Being posetal, $\Qt^{\mathcal S}$ 
is also locally Caretsian closed.  In addition, if $\Qt$ is (locally) 
(w/f)-Cartesian closed then, by putting $Z=C$ in $(**)$ we see that 
$\Qt^{\mathcal S}$ is also (locally) Cartesian closed. 

It follows that, replacing $\Qt$ with $\Qt^{\mathcal S}$  we obtain a locally 
Cartesian closed posetal model category all of whose trivial co-fibrant objects 
are small, in the sense that whenever $\perp\rightwc X\rightf Y$ the fibration 
$X\rightf Y$ is small. Of course, the model category $\Qt^{\mathcal S}$ may be 
less interesting than the original category $\Qt$. But  the above argument shows 
that - at least for posetal model categories all of whose morphisms  are 
co-fibrations - it is possible to have a notion of smallness which corresponds 
exactly to a fibration $X\rightf Y$ being small when $\perp\rightwc X$. Since, 
as explained above, there cannot be a natural category theoretic definition of 
smallness capturing precisely Voevodsky's notion of small fibration, we believe 
that, given the level of generality we are working in, the above is as good an 
approximation of this notion as could be expected. 

We conclude that: 

\begin{prp}\label{main}
 Let $\Qt$ be a posetal model category all of whose morphisms are 
co-fibrations. Let $\perp \rightwc \tilde U \rightf \top$, and define  a 
fibration $Y\rightf X$ to be small if $\perp \rightwc Y$. Then the fibration 
$p:\tilde U\lra \top$ is universal. If, in addition, $\Qt$ is locally 
(w/f)-Cartesian closed then $\Qt$ meets the Univalence Axiom, with respect to 
the above notion of small fibrations. 
\end{prp}

In the next section we give an example of a non-trivial model category 
satisfying all the assumptions of Proposition \ref{main}. 

\section{The model category $\Qtc$}\label{qtc}

The main result of \cite{GaHa} is the construction of a (non-trivial) posetal 
model category of classes of sets, $\QtN$.  In this section we show that the 
full sub-category, $\Qtc$, of all co-fibrant objects meets all the assumptions 
of Proposition \ref{main}. Indeed, this follows almost immediately from the 
results of \cite{GaHa}, but for the sake of completeness, we give a 
self-contained proof. 

To simplify the exposition, and in order to avoid irrelevant foundational 
issues, we give a slightly simplified version of the model category $\Qtc$. Let 
$\Qtc$ be the category whose objects are the members of $\mathbb P( \mathbb P 
(\mathbb N)):=\{X\subseteq \{M\subseteq \mathbb N\}\}$ and for $X,Y\in \Ob\Qtc$ 
let $X\lra Y$ precisely when for every $x\in X$ there exists $y\in Y$ such that 
$x\subseteq y$. We leave it as an easy exercise for the reader to verify that 
this is indeed a (posetal) category. 

\begin{claim}
 The category $\Qtc$ has limits. Direct limits are given by unions $X\vee Y = 
X\cup Y$, and inverse limits are given by pointwise intersection, namely 
$X\times Y = \{x\cap y: x\in X, y\in Y\}$. The same formulas hold for infinite 
limits. 
\end{claim}
\begin{proof}
 This is straightforward. Assume, e.g. that we are given $X,Y$ and $Z\lra X$, 
$Z\lra Y$. By definition, this means that for all $z\in Z$ there are $x\in X$, 
$y\in Y$ such that $z\subseteq x$ and $z\subseteq y$. This means that for all 
$z\in Z$ there are $x\in X$ and $y\in Y$ such that $z\subseteq x\cap y$. This 
proves that $X\times Y$ as defined above is the inverse limit of $X$ and $Y$. 
The proof for direct limits is similar. 
\end{proof}

Now we endow $\Qtc$ with a model structure. We do not attempt to justify the 
intuition behind these definitions - this is done in some detail in 
\cite{GaHa}. In order to meet the assumptions of Proposition \ref{main}, we 
must require that all morphisms are labelled (c). So we now proceed to the (w) 
and (f) labels. 

For the definition of weak equivalences it is convenient to denote for $X, Y\in 
\Ob \Qtc$, $X\lra^* Y$ if for all $x\in X$ there exists $y\in Y$ such that 
$|y\setminus x|<\aleph_0$. We now set $X\rightw Y$ if $X\lra Y$ and $Y\lra^* 
X$. This definition obviously satisfies Axiom  (M5) (2 out of 3). Also, if 
$Z\xleftarrow{(wc)} X\lra Y$ then $Y\rightwc Z\vee Y$. Indeed, if $r\in Z\vee 
Y$ then either $r\in Y$ in which case $r\setminus r=\0$ or $r\in Z$, in which 
case there is $x\in X$ such that $|r\setminus x|<\aleph_0$ but $X\lra Y$, so 
there is $y\in Y$ such that $x\subseteq y$ and $|r\setminus y|\le |r\setminus 
x|<\aleph_0$. This shows that the (wc)-part of Axiom (M4) is met by this 
notation. 

It remains to define the (f)-labelling: an arrow $X\lra Y$ is labelled (f) if 
and only if for every $x\in X\cup\{\emptyset\}$, $y\in Y$ and a finite
subset $\{b_1,\dots ,b_n\}\subseteq  y$ there exists $x'\in X$ such that 
$(x\cap y)\cup\{b_1,...,b_n\}\subseteq x'$. 

First, we observe: 

\begin{claim}
 If $X\xrightarrow{(wcf)} Y$ then $Y\lra X$. 
\end{claim}
\begin{proof}
 Let $y\in Y$. We have to show that there exists $x\in X$ such that $y\subseteq 
x$. Let $x_0\in X$ be such that $z:=y\setminus x$ is finite, as provided by the 
(w)-label. So the (f)-label, applied for $x_0,y$ and $z\subseteq y$ assures the 
existence of $x$ with the desired property. 
\end{proof}
This claim gives us, automatically, one part of (M1) - any arrow right-lifts 
with respect to an isomorphism - one part of (M2) - any arrow $X\rightc Y$ 
decomposes as $X\rightc Y\rightwf Y$ and (M3) (it remains only to verify that 
fibrations are stable under base-change). Axiom (M4) is also automatic. So we 
are left with the $(wc)\rtt (f)$ part of (M1), the (wc)-(f) decomposition of 
(M2) and the stability of fibrations under base change. All computations are 
trivial, so we will be brief. 

Let $X\rightwc Y$ and $W\rightf Z$ be such that $X\lra W$ and $Y\lra Z$. We 
have to show that $Y\lra W$. So let $y\in Y$. Let $x\in X$ be such that 
$b:=y\setminus x$ is finite. Let $w\in W$ be such that $x\subseteq w$. Let 
$z\in Z$ be such that $y\subseteq z$. Apply the definition of (f)-arrows with 
respect to $w,z$ and $b$. Then there exists $w'\in W$ such that $(w\cap z)\cup 
b\subseteq w'$. So $y\subseteq w'$, as required. An essentially similar 
argument shows that fibrations are stable under base-change. 

To prove (M2), let $X\lra Y$ be any arrow. Let 
\[
X_{wc}:=\{x\cup y_0: x\in X, (\exists y\in Y)(y_0\subseteq y), y_0\text{ 
finite}\}.
\]
Then $X\rightwc X_{wc}\rightf Y$, as can be readily checked. 

We conclude that $\Qtc$ is a posetal model category all of whose arrows are 
co-fibrations. It is not trivial (in the sense that not all arrows are 
fibrations) because $\0\rightwc X$ is not a fibration unless $X=\0$. Since it 
is posetal, to show that it is locally Cartesian closed, it suffices to show 
that it is Cartesian closed.

Define, for $C,B\in \Ob\Qtc$: 
\[
C^B:=\bigcup \{ Z: Z\times B \lra C ,\, A\lra Z\}.
\]
This is, obviously, an object in $\Qtc$, so we need only check that 
$C^B \times B \lra C$ and that for every
object $Z$,  $Z\lra A$ and $Z\times B \lra C$ implies  $Z\lra C^B$.
The latter is immediate by definition of $C^B$. The former 
requires a little argument. We need to check that for every $d\in C^B$ 
and $b\in B$ there is a morphism  $\{d\cap b\}\lra C$. By definition of $C^B$,
there exists $Z$ such that $Z\times B \lra C$ and $d\in Z$,  i.e. $\{d\}\lra Z$.
By definition of the product $Z\times B$, this implies $\{d \cap b\} \lra C$,
as required.

\begin{rem}
 Note that the above shows that $\Qtc$ is, in particular, a logical model 
category in the sense of \cite[Definition 23]{ArKa}. Consequently (cf. Theorem 
26) $\Qtc$ admits a sound interpretation of the syntax of type theory (though 
the lack of non-trivial sections probably makes this interpretation trivial). 
\end{rem}

All of the above shows that $\Qtc$ is a posetal locally Cartesian closed model 
category which is non-trivial (in the sense that not all morphisms are labelled 
(fc)). So in order to apply Proposition \ref{main} it remains to show that it 
is locally (w/f)-Cartesian closed. We prove: 

\begin{claim}\label{claim5}
$ Z\times B\rightwc Z\times C$  iff for all $\{z\}\lra Z$, $\{z\}\times 
B\rightwc \{z\}\times C$
\end{claim}

\begin{proof}
The right to left direction is immediate from the definition of (wc)-arrows, so 
we prove the other direction. The arrow $Z\times B\rightwc Z\times C$ means 
that: 
\begin{itemize}
 \item for any $z \in Z$, $b \in B$ exists $z' \in Z$ and $c'\in C$
such that $\{z \cap b\}\lra \{z'\cap c'\}$; and 
\item for any $z''\in Z$, $c''\in C$ exists $z\in Z$, $b\in B$ such that
$ \{z''\cap c''\}\lra^* \{z \cap b\}$. 
\end{itemize}
Observe that the first bullet (for fixed $z\in Z, b\in B$) gives $z \cap 
b\subseteq z' \cap c'$, implying that $z \cap b\subseteq z \cap z'\cap 
c'\subseteq z \cap c'$, therefore $\{z\}\times B \lra\{z\}\times C$.

Analogously, for fixed $z''\in Z, c''\in C$  the assumption $ \{z'' \cap 
c''\}\lra^* \{z \cap b\}$
implies $\{z'' \cap c''\}\lra^* \{z'' \cap z \cap b\}\lra \{z''\cap b\}$. 
Combining these two observations we get $\{z\}\times B \rightwc \{z\} \times 
C$. 
\end{proof}

Now, given $A\in \Ob\Qtc$ and $B\lra A$, $C\lra A$, we define 
\[
(C^B_w)/A = \bigcup \{Z : Z\times B \rightw Z\times C \}\times A 
\] 
and show that this is an object representing $Hom^{(w)}_A(-\times B,C)$ ($\Qtc$ 
is trivially right proper, so this is indeed a functor). More precisely: 
\begin{claim}
 For all $Z\lra A$, we have 
                     $Z \lra (C^B_w)/A$ if and only if  $Z\times B \rightw  
Z\times C.$
\end{claim} 
\begin{proof}
The Right to left direction is immediate from the definition. So suppose $Z\lra 
(C^B_w)/A$. We need to show that $Z\times B \rightw  Z\times C$. 
By Claim \ref{claim5}, this happens if for all $\{z\}\lra  Z$, 
$\{z\}\times B \rightw \{z\}\times C$.
But our assumption that $Z\lra (C^B_w)/A$ implies that $Z \lra Z'$ for some 
$Z'$ such that $Z'\times B 
\rightw Z'\times C$. So $\{z\}\lra Z'$, by Claim \ref{claim5} we are done.
\end{proof}

Combining everything together we get: 

\begin{theorem}
 There exists a non-trivial posetal model category satisfying the Univalence 
Axiom. 
\end{theorem}

Before continuing we remark that $\Qtc$, as presented in this note is a full 
sub-category of the category of co-fibrant objects in the model category $\QtN$ 
defined in \cite{GaHa}. The exact same proof of the above theorem would work 
for the full category of co-fibrant objects in $\QtN$. Of course this category 
captures the full homotopy structure of $\QtN$, and may - therefore - be a more 
interesting example. We remark also that there does not seem to be anything 
spcial about $\mathbb N$ or about $\aleph_0$ in the above construction (or in 
the more genreal construction of $\QtN$).  Apparently, the exact same 
construction could be achieved for any regular cardinal $\lambda$ (in place of 
$\aleph_0$) replacing, throughout ``finite'' by ``less than $\lambda$''. This 
gives --- in addition to the formal discussion of smallness in the previous 
sub-section ---  an analogy with Voevodsky's notion of small fibrations: it is 
not unreasonable (see the following paragraph) to think of the morphisms in the 
resulting model category as a class of injections (satisfying certain 
compatibility conditions), our definition of smallness implies that a fibration 
is small precisely when every memebr of the class of these injections has a 
domain smaller than $\lambda$. 

To conclude, let us consider the category $\FC$, whose objects are $\Ob \Qtc$ 
and such that $\Mor(X,Y)$  consists of the arrows $X\xrightarrow{\sigma} Y$ for 
$X,Y\in \Ob\FC$ such that $\sigma:\bigcup X\lra \bigcup Y$ and $\sigma(X)\lra 
Y$ is an arrow in $\Mor\Qtc$ (where $\sigma(X):=\{\,\{\sigma(a):a\in x\}\,:\,x\in X\}$). The 
category $\FC$ is, on the one hand, obvioulsy richer than $\Qtc$ (it is not 
posetal). But, on the other hand, it is readily seen that any slice of $\FC$ is 
(naturally) equivalent to the corresponding slice of $\Qtc$. This local model 
structure induces naturally a (c)-(f)-(w) labeling on $\Mor(\FC)$ (see 
\cite{GaHa} for the details) satsifying Quillen's axioms (M1)-(M5). But the 
category $\FC$ does not have products and co-products. So we ask: \\

\noindent{\bf Question}: Is there a model category $\FC'$ such that the labeled 
category $\FC$ described above embeds in $\FC'$? Does $\FC'$ satisfy the 
Univalence Axiom? \\

\noindent\emph{Acknowledgement} We would like to thank N. Durov for his help in 
the definition of (w/f)-Cartesian closed model categories. 

\bibliographystyle{alpha}
\bibliography{../../Bibfiles/harvard}

\begin{thebibliography}{Awo10}

\bibitem[Awo10]{Awodey}
Steve Awodey.
\newblock {\em Category theory}, volume~52 of {\em Oxford Logic Guides}.
\newblock Oxford University Press, Oxford, second edition, 2010.

\bibitem[Gar11]{Ob}
Richard Garner, editor.
\newblock {\em Mini workshop: The homotopy interpretation of constructive type
  theory}, Oberwolfach meeting Report No. 11/2011, 2011.

\bibitem[GH10]{GaHa}
M.~Gavrilovich and Assaf Hasson.
\newblock Exercises de style: a homotopy theory for set theory {I}.
\newblock Available on arXiv, 2010.

\bibitem[GK95]{ArKa}
T.~E. Gallivan and Arie Kapulkin.
\newblock Failure of universality in noncompact lattice field theories.
\newblock {\em J. Math. Phys.}, 36(5):2341--2353, 1995.

\bibitem[Qui67]{Qui}
Daniel~G. Quillen.
\newblock {\em Homotopical algebra}.
\newblock Lecture Notes in Mathematics, No. 43. Springer-Verlag, Berlin, 1967.

\bibitem[Voe10]{Voev}
Vladimir Voevodsky.
\newblock Univalent foundations project (a modified version of an nsf grant
  application).
\newblock Available at
  http://www.math.ias.edu/$\sim$vladimir/Site3/Univalent$\underline{\;\;}$Foun%
dations.html, 2010.

\bibitem[War]{Warren}
Michael Warren.
\newblock Univalen axiom as a property of a model category?
\newblock
  http://mathoverflow.net/questions/73086/univalent-axiom-as-a-property-of-a-m%
odel-category.

\end{thebibliography}
\end{document}